\documentclass[11pt, oneside]{amsart}   	% use "amsart" instead of "article" for AMSLaTeX format
\usepackage{geometry}                		% See geometry.pdf to learn the layout options. There are lots.=
\usepackage{graphicx}				% Use pdf, png, jpg, or eps§ with pdflatex; use eps in DVI mode
\usepackage{amsmath}
\usepackage{euscript}
\usepackage{amsfonts}
\usepackage{amssymb}
\usepackage{amsthm}
\usepackage{tikz}
\usepackage{graphicx}				% Use pdf, png, jpg, or eps§ with pdflatex; use eps in DVI mode
\theoremstyle{plain}% default
\newtheorem{thm}{Theorem}[section]

\newtheorem{Thm}{Theorem}

\theoremstyle{definition}

\theoremstyle{remark}

\DeclareMathOperator{\Hom}{Hom}
\DeclareMathOperator{\Endo}{End}

\DeclareMathOperator{\Ext}{Ext}

\DeclareMathOperator{\Rep}{Rep}

\DeclareMathOperator{\N}{\mathbb{N}}
\DeclareMathOperator{\R}{\mathbb{R}}

\DeclareMathOperator{\kdim}{dim}

\title{Universal Deformation Rings and Quivers of Finite Representation Type}
\author{Roberto C. Soto and Daniel J. Wackwitz}
\begin{document}
\maketitle

\begin{abstract}
Let $k$ be a field of arbitrary characteristic and let $Q$ be a quiver of finite representation type.  In this paper we prove that if $M$ is an indecomposable $kQ$-module then the universal deformation ring of $M$ over $kQ$ is isomorphic to $k.$
\end{abstract}
%%%%%%%%%%%%%%%%%%%%%%%%%%%%%%%%%%%%%%%%%%
\section*{Introduction}
%%%%%%%%%%%%%%%%%%%%%%%%%%%%%%%%%%%%%%%%%%

Universal deformations rings of representations of profinite Galois groups have played an important role in proving long-standing problems in number theory.  The most celebrated of these is Fermat's Last Theorem (see \cite{modular1997forms}) but many other results have also been established such as the Taniyama-Shimura-Weil and Serre conjectures (see \cite{bockle2007present, shimura2001taniyama}).  In general, if $k$ is a field of arbitrary characteristic, $W$ is a complete local commutative Noetherian ring with residue field $k,$ $G$ is a profinite group and $V$ is a finite dimensional $k$-vector space with continuous $G$ action, then it has been shown (see \cite{deSmit1997lenstra, mazur1989deforming}) that if $\Endo_{kG} (V) \cong k$ and $H^1(G, \Endo_k(V))$ is finite dimensional over $k,$ that $V$ has a universal deformation ring, $R_W(G, V).$  Moreover, in \cite{deSmit1997lenstra} it was shown that $R_W(G, V)$ is isomorphic to the inverse limit of the universal deformation rings of $R_W(G_i, V),$ where the $G_i$ range over all discrete finite quotients of $G$ through which the action of $G$ on $V$ factors.  As such it is important that we understand universal deformation rings for modules over finite dimensional algebras $kG$, where $G$ is a finite group.  Indeed deformations of modules for finite dimensional algebras have been studied by many authors in different contexts (see \cite{Ile2004change, laudal2002noncom, yau2005deftheory}) and in particular we use the work of Velez in \cite{velez2012frobenius} to prove our result.

In this paper we consider $k$-algebras arising from quivers of finite representation type.  In Section 1 we define universal deformation rings and when they exist.  In Section 2 we give a brief overview of quivers of finite representation type. Finally in Section 3 we discuss Gabriel's Theorem as it is a major result that allows us to prove our theorem below in Section 4.

\begin{Thm} \label{main}
Let $Q$ be a quiver of finite representation type and  $M$  an indecomposable $kQ$-module.  Then the universal deformation ring of $M$ over $kQ$ is isomorphic to $k.$
\end{Thm}

%%%%%%%%%%%%%%%%%%%%%%%%%%%%%%%%%%%%%%%%%%
\section{Universal deformation rings}
%%%%%%%%%%%%%%%%%%%%%%%%%%%%%%%%%%%%%%%%%%

Let $k$ be a field of arbitrary characteristic.  Let $\hat{\mathcal{C}}$ be the category which has all complete local commutative Noetherian $k$-algebras with residue field $k$ as objects, and continuous $k$-algebra homomorphisms inducing the identity map on $k$ as morphisms.

Suppose $\Lambda$ is a finite dimensional $k$-algebra, $V$ is a finitely generated $\Lambda$-module and $R$ is an object in $\hat{\mathcal{C}}$.  A \emph{lift} of $V$ over $R$ is a pair $(M,\phi)$ where $M$ is a finitely generated $R \otimes_k \Lambda$-module which is free over $R$ and $\phi$ is a $\Lambda$-module isomorphism $\phi: k \otimes_R M \rightarrow V$. Two lifts $(M,\phi)$ and $(M',\phi')$ of $V$ over $R$ are said to be isomorphic if there exists an $R \otimes_k \Lambda$-module isomorphism $f:M \rightarrow M'$ such that $\phi' \circ (k \otimes_R f) = \phi$.  A \emph{deformation} $[M,\phi]$ of $V$ over $R$ is the isomorphism class of the lift $(M,\phi)$. Define Def$_{\Lambda}(V,R)$ to be the set of all such deformations of $V$ over $R$.  Define the \emph{deformation functor} $F_V:\hat{\mathcal{C}} \rightarrow$ Sets as the covarient functor which sends a ring $R$  in $\hat{\mathcal{C}}$ to Def$_{\Lambda}(V,R)$ and a morphism $\alpha:R \rightarrow R'$ in $\hat{\mathcal{C}}$ to the set map $F_V(\alpha):\text{Def}_{\Lambda}(V,R) \rightarrow \text{Def}_{\Lambda}(V,R')$, which sends $[M,\phi]$ to $[R\otimes_{R,\alpha} M, \phi_{\alpha}]$ where $\phi_{\alpha}$ is the composition $k\otimes_{R'} (R' \otimes_{R,\alpha} M) \cong k \otimes_R M \xrightarrow{\phi} V$. Define the \emph{tangent space} of $F_V$ to be the set $t_V = F_V(k[\epsilon])$ where $k[\epsilon]$ denotes the ring of dual numbers over $k$, ie $\epsilon^2 = 0$.

The functor $F_V$ is said to be represented by a ring $R(\Lambda,V)$ in $\hat{\mathcal{C}}$ if $F_V$ is naturally isomorphic to the functor $Hom_{\Lambda}(R(\Lambda,V),-)$. In other words, there exists a lift $(U(\Lambda,V), \phi_U)$ of $V$ over $R(\Lambda, V)$ such that for any $R$ in $\hat{\mathcal{C}}$, the map $\nu_R: \text{Hom}_{\hat{\mathcal{C}}}(R(\Lambda,V),R) \rightarrow F_V(R)$, which sends $\alpha \in \text{Hom}_{\hat{\mathcal{C}}}(R(\Lambda,V),R)$ to $F_V(\alpha)([U(\Lambda,V),\phi_U])$, is bijective. In this case, $R(\Lambda,V)$ is said to be the \emph{universal deformation ring of $V$ over $R$}.

\begin{thm}
\label{UDRexist}
There is a $k$-vector space isomorphism $t_V \cong \Ext_{\Lambda}^1(V,V)$. Furthermore, when $\Endo_{\Lambda}(V) \cong k$, then $V$ has a universal deformation ring $R(\Lambda, V)$.
\end{thm}
\begin{proof}
See \cite[Proposition 2.1]{velez2012frobenius}.
\end{proof}

\begin{thm}
\label{UDRform}
If $\Endo_{\Lambda}(V) \cong k$ and $\kdim_k \Ext^1_{\Lambda} (V,V)=r,$ then there exists a surjective homomorphism $\lambda: k[[t_1, \cdots, t_r]] \rightarrow R(\Lambda,V)$ in $\hat{\mathcal{C}}$, and $r$ is minimal with this property.
\end{thm}
\begin{proof}
Since End$_{\Lambda}(V) \cong k$, $F_V$ is representable by Theorem \ref{UDRexist}.  Therefore $t_V = F_V(k \left[\epsilon \right]) \cong Hom_{\hat{\mathcal{C}}} (R(\Lambda,V), k \left[ \epsilon \right])$. Furthermore, by Theorem \ref{UDRexist}, $t_V \cong Ext^1_{\Lambda} (V,V)$ as a $k$-vector space, thus giving the desired result.
\end{proof}

%%%%%%%%%%%%%%%%%%%%%%%%%%%%%%%%%%%%%%%%%%
\section{Quivers of Finite Representation Type}
%%%%%%%%%%%%%%%%%%%%%%%%%%%%%%%%%%%%%%%%%%

A  quiver is a quadruple $Q = (Q_0, Q_1, s, t),$ where $Q_0$ is a finite set of vertices, $Q_1$ is a finite set of arrows, and $s, t: Q_1 \rightarrow Q_0$ are maps assigning to each arrow it \emph{source}, resp. \emph{target}.  A \emph{representation} $M$ of a quiver $Q$ consists of a family of vector spaces $V_i$ indexed by the vertices $i \in Q_0,$ together with a family of linear maps $f_\alpha: V_{s(\alpha)} \rightarrow V_{t(\alpha)}$ indexed by the arrows $\alpha \in Q_1.$  If $M = (V_i, f_\alpha)$ is a representation of $Q,$ then its \emph{dimension vector} $\underline{n} = \underline{\kdim} M= (\kdim V_i)_{i \in Q_0}.$

Given two representations $M = \left((V_i)_{i \in Q_0}, (f_{\alpha})_{\alpha \in Q_1} \right), N = ((W_i)_{i \in Q_0}, (g_\alpha)_{\alpha \in Q_1})$ of a quiver $Q,$ a morphism $u: M \rightarrow N$ is a family of linear maps $(u_i: V_i \rightarrow W_i)_{i \in Q_0}$ such that for any $\alpha \in Q_1$ we have $u_{t(\alpha)} \circ f_\alpha = g_\alpha \circ u_{s(\alpha)}.$ %INCLUDE COMMUTING DIAGRAM HERE
For a quiver $Q$ and a field $k$ we can form the category $\Rep_k(Q)$ whose objects are representations of $Q$ with the morphisms as defined above.  A morphism $\phi: M = (V_i, f_\alpha) \rightarrow N = (W_i, g_\alpha)$ is an \emph{isomorphism} if $\phi_i$ is invertible for every $i \in Q_0.$  As can be expected, we wish to classify all representations of a given quiver $Q$ up to isomorphism.  

If $M$ and $N$ are two representations of the same quiver $Q,$ we define their \emph{direct sum} $M \oplus N$ by $(M \oplus N)_i = V_i \oplus W_i$ for all $i \in Q_0,$ and $(M \oplus N)_\alpha: V_{s(\alpha)} \oplus W_{s(\alpha)} \rightarrow V_{t(\alpha)} \oplus W_{t(\alpha)}$ for all $\alpha \in Q_1.$   We say that $M$ is a \emph{trivial representation} if $V_i = 0$ for all $i \in Q_0.$  If $M$ is isomorphic to a direct sum $M_1 \oplus M_2,$ where $M_1$ and $M_2$ are nontrivial representations, then $M$ is called \emph{decomposable}.  Otherwise $M$ is called \emph{indecomposable}.  Every representation has a unique decomposition into indecomposable representations, up to isomorphism and permutation of components.  Thus the classification problem reduces to classifying the indecomposable representations.  We say that a quiver is of \emph{finite representation type}, or just \emph{finite type}, if it has only finitely many indecomposable representations; otherwise, it is of \emph{infinite representation type}.  The following theorem classifying the quivers of finite representation type is due to Gabriel (see \cite{gabriel1972finite}, \cite{kraftRiedtmann1985geom}).

\begin{thm}[Gabriel]\label{gabrielweak}
A quiver is of finite representation type if and only if each connected component of its underlying undirected graph is a Dynkin graph of type $A, D,$ or $E,$ shown below:
\end{thm}
\begin{center}
		%%%%%%%%%%%%%%%%%%%%%%%%%%%%%%%%%%%%%%%%%%%%%
		%	 This is ZD_infty
		%%%%%%%%%%%%%%%%%%%%%%%%%%%%%%%%%%%%%%%%%%%%%
\begin{tikzpicture}[scale =1.25]
%\draw [help lines] (-1,-1) grid (10,1);

% Graph A_n
\node at (0, 7) {$A_n:$};
\foreach \x in {1,...,5}
  \node at (\x, 7) {$\bullet$};
  
\foreach \x in {1, 3, 4}
  \draw[-] (\x+0.15, 7) -- (\x+.85, 7);

\draw[dashed] (2.15, 7) -- (2.85, 7);

\node[right] at (5.5, 7) {($n$ vertices, $n \geq 1$)};

% Graph D_n
\node at (0, 6) {$D_n:$};
\foreach \x in {1,...,5}
  \node at (\x, 6) {$\bullet$};
  
\foreach \x in {1, 3, 4}
  \draw[-] (\x+0.15, 6) -- (\x+.85, 6);

\draw[dashed] (2.15, 6) -- (2.85, 6);

\node[right] at (5.5, 6) {($n$ vertices, $n \geq 4$)};

\node at (4, 5) {$\bullet$};
\draw[-]   (4, 5.15)--(4,5.85);

% Graph E_6
\node at (0, 4) {$E_6:$};
\foreach \x in {1,...,5}
  \node at (\x, 4) {$\bullet$};
  
\foreach \x in {1, ..., 4}
  \draw[-] (\x+0.15, 4) -- (\x+.85, 4);

\node at (3, 3) {$\bullet$};
\draw[-]   (3, 3.15)--(3,3.85);

% Graph E_7
\node at (0, 2) {$E_7:$};
\foreach \x in {1,...,6}
  \node at (\x, 2) {$\bullet$};
  
\foreach \x in {1, ..., 5}
  \draw[-] (\x+0.15, 2) -- (\x+.85, 2);

\node at (4, 1) {$\bullet$};
\draw[-]   (4, 1.15)--(4,1.85);

% Graph E_8
\node at (0, 0) {$E_8:$};
\foreach \x in {1,...,7}
  \node at (\x, 0) {$\bullet$};
  
\foreach \x in {1, ..., 6}
  \draw[-] (\x+0.15, 0) -- (\x+.85, 0);

\node at (5, -1) {$\bullet$};
\draw[-]   (5, -0.15)--(5,-0.85);

 \end{tikzpicture}
\end{center}

%%%%%%%%%%%%%%%%%%%%%%%%%%%%%%%%%%%%%%%%%%
\section{Gabriel's Theorem}
%%%%%%%%%%%%%%%%%%%%%%%%%%%%%%%%%%%%%%%%%%

The \emph{algebra of a quiver} $Q$ is the associative algebra $kQ$ determined by the generators $e_i$ and $\alpha,$  where $i \in Q_0$ and $\alpha \in Q_1,$ and the relations
	$$e_i^2 = e_i, e_i e_j = 0 ( i \neq j), e_{t(\alpha)}\alpha = \alpha e_{s(\alpha)} = \alpha.$$  
Note that the category of representations of any quiver $Q$ is equivalent to the category of left $kQ$-modules (see \cite[Theorem III.1.5]{ausreitsmal1997reptheory}).

Recall, if $M$ and $N$ are arbitrary $kQ$-modules we define the groups $\Ext_Q^i (M, N)$ as follows:  first choose a projective resolution
	$$\cdots \longrightarrow P_2 \longrightarrow P_1 \longrightarrow P_0 \longrightarrow M \longrightarrow 0.$$
Then take morphisms to $N$ yielding a complex
	$$\Hom_Q(P_0, N) \longrightarrow \Hom_Q(P_1, N) \longrightarrow \Hom_Q(P_2, N) \longrightarrow \cdots$$
The homology groups of this complex are independent of the choice of a projective resolution of $M;$  the $i$th homology group is denoted $\Ext_Q^i(M, N)$ (see \cite[Section 2.4]{benson1998cohom}).

Note that $\Ext_Q^0(M, N) = \Hom_Q(M, N)$ and recall that $\Ext_Q^1(M, N)$ is the set of equivalence classes of \emph{extensions} of $M$ by $N,$ i.e., of exact sequences of $kQ$-modules 
	$$0 \longrightarrow N \longrightarrow E \longrightarrow M \longrightarrow 0$$
up to isomorphisms that induce the identity maps on $N$ and $M$ (see \cite[Section 2.6]{benson1998cohom}).

% Note that every exact sequence of finite-dimensional representations
%	$$ 0 \longrightarrow M' \longrightarrow M \longrightarrow M'' \longrightarrow 0$$
% satisfies
%	$$\underline{\kdim} M = \underline{\kdim} M' + \underline{\kdim} M''.$$
% We also have that any two isomorphic finite-dimensional representations have the same dimension vector.  Naturally, a central problem of quiver theory is to describe the isomorphism classes of finite-dimensional representations of a prescribed quiver, having a prescribed dimension vector.

The \emph{Euler form} of the quiver $Q$ is the bilinear form $\langle , \rangle_Q$ on $\R^{Q_0}$ given by 
	$$\langle \underline{m}, \underline{n} \rangle_Q = \sum_{i \in Q_0} m_i n_i = \sum_{\alpha \in Q_1} m_{s(\alpha)} n_{t(\alpha)}$$
for any $\underline{m} = (m_i)_{i \in Q_0}$ and $\underline{n} = (n_i)_{i \in Q_0}.$  The quadratic form associated to the Euler form is called the \emph{Tits form} $q_Q$, i.e.
	$$q_Q(\underline{n}) = \langle \underline{n}, \underline{n} \rangle_Q = \sum_{i \in Q_0} n_i^2  - \sum_{\alpha \in Q_1} n_{s(\alpha)} n_{t(\alpha)}.$$  
Note that the Tits form depends only on the underlying undirected graph of $Q$ and determines the graph uniquely (see \cite{brion2012repquiver}).

Finally a representation $M$ of the quiver $Q$ is called a \emph{Schur representation} (also known as a \emph{brick}), if $\Endo_Q(M) \cong k.$  Clearly, any Schur representation is indecomposable, but the converse only holds if the Tits form of $Q$ is positive definite.

Now we obtain a more precise form of Gabriel's Theorem with important consequences for our result.
\begin{thm} \label{gabriel}
Assume that the Tits form $q_Q$ is positive definite.  Then:
\begin{itemize}
	\item[(i)]	Every indecomposable representation is Schur and has no non-zero self-extensions.
	\item[(ii)]	The dimension vectors of the indecomposable representations are exactly those $\underline{n} \in \N^{Q_0}$ such that $q_Q(\underline{n}) = 1.$
	\item[(iii)]	Every indecomposable representation is uniquely determined by its dimension vector, up to isomorphism.
	\item[(iv)]	There are only finitely many isomorphism classes of indecomposable representations of $Q.$
\end{itemize}
\end{thm}

\begin{proof}
See \cite[Theorem 2.4.3]{brion2012repquiver}.
\end{proof}

Note that, as a result of \cite[Proposition 1.4.6]{brion2012repquiver}, the quivers $Q$ with positive definite Tits form are exactly the quivers which are of finite representation type. In particular we have that if $Q$ is a quiver of finite representation type and $M$ is an indecomposable representation of the quiver $Q$, i.e. an indecomposable $kQ$-module, then $\Endo_Q(M) \cong k$ and $\Ext_Q^1(M, M) = 0.$ 

%%%%%%%%%%%%%%%%%%%%%%%%%%%%%%%%%%%%%%%%%%
\section{Proof of Theorem \ref{main}}
%%%%%%%%%%%%%%%%%%%%%%%%%%%%%%%%%%%%%%%%%%

Now we can prove our result.  Assume the hypotheses of Theorem \ref{main}.
\begin{proof}
By Theorem \ref{gabriel} we have that $\Endo_{Q}(M) \cong k$ and that $\Ext_Q^1(M, M) = 0.$ Thus by Theorem \ref{UDRexist} a universal deformation ring exists and by Theorem \ref{UDRform} we have a surjective map $\lambda: k \rightarrow R(kQ,M).$  Since  $\ker \lambda = 0$ we obtain our isomorphism.
\end{proof}

% \nocite{*}
\bibliographystyle{amsplain}
\bibliography{References} %NO SPACES WHEN SAVING BIB FILE
% CITE KEYS

% ausreitsmal1997reptheory
% benson1998cohom
% brion2012repquiver
% bockle2007present
% deSmit1997lenstra
% gabriel1972finite
% Ile2004change
% laudal2002noncom
% kraftRiedtmann1985geom
% mazur1989deforming
% modular1997forms
% schiffler2014quiverreps
% shimura2001taniyama
% velez2012frobenius
% yau2005deftheory
    
\end{document}